\documentclass[dvipsnames]{siamart171218} 
\UseRawInputEncoding 

\usepackage{graphicx}        
\usepackage{xcolor}
\usepackage{amsmath}
\usepackage{comment}

\usepackage[leqno]{mathtools}

\newcommand\fullwidthdisplay{\displayindent0pt \displaywidth\columnwidth}
\AtBeginDocument{
  \everydisplay\expandafter{\expandafter\fullwidthdisplay\the\everydisplay}
}

\usepackage{amssymb}
\usepackage{mathrsfs}
\usepackage[font=small,labelfont=bf]{caption}
\usepackage[shortlabels]{enumitem}
\usepackage{chngcntr}
\usepackage{bm} 
\usepackage{algorithmic,algorithm}
\usepackage[hang,flushmargin]{footmisc} 

\usepackage{graphicx}
\usepackage{dcolumn}
\usepackage{bm}
\usepackage{amsmath,amssymb}
\usepackage{makeidx}
\usepackage{amsfonts}
\usepackage{comment}
\usepackage[usenames,dvipsnames]{pstricks}
\usepackage{subfigure}
\usepackage{braket}
\usepackage{float}

\newtheorem{assumption}[theorem]{Assumption}
\newtheorem{remark}[theorem]{Remark}

\newtheorem{problem}[theorem]{Problem}

\numberwithin{theorem}{section}
\numberwithin{equation}{section}

\renewcommand{\cal}[1]{\mathcal{#1}}

\renewcommand{\r}{\mathbb{R}}

\newcommand{\n}{\mathbb{N}}

\newcommand{\zp}{\mathbb{Z}_+}
\newcommand{\mmag}[1]{\left|#1\right|}
\newcommand{\s}{\mathcal{S}}

\renewcommand{\L}{\mathcal{L}}
\newcommand{\e}{\mathcal{E}}
\newcommand{\x}{\mathcal{X}}
\newcommand{\y}{\mathcal{Y}}

\newcommand{\twx}{2^\mathcal{X}}

\newcommand{\supp}[1]{\operatorname{supp}\left(#1\right)}
\newcommand{\norm}[1]{\left|\left|{#1}\right|\right|}

\newcommand{\Pbl}[1]{\mathbb{P}_\gamma\left(#1\right)} %
\newcommand{\Ebl}[1]{\mathbb{E}_\gamma\left[#1\right]} %
\newcommand{\Pbp}[1]{\mathbb{P}_\pi\left(#1\right)}

\headers{Approximations of countably-infinite linear programs}{Juan Kuntz, Philipp Thomas, Guy-Bart Stan, and Mauricio Barahona}
\title{Approximations of countably-infinite linear programs over bounded measure spaces\thanks{JK was supported by a BBSRC PhD Studentship (BB/F017510/1). GBS acknowledges support by an EPSRC Fellowship for Growth (EP/M002187/1) and the EU H2020-FETOPEN-2016-2017 project 766840 (COSY-BIO). MB acknowledges support from EPSRC grant EP/N014529/1 supporting the EPSRC Centre for Mathematics of Precision Healthcare.}
}

\author{
  Juan Kuntz\thanks{Department of Mathematics and Department of Bioengineering, Imperial College London, London SW7 2AZ, United Kingdom.  \, \textit{Current address:} Department of Statistics,  University of Warwick, Coventry, CV4 7AL,  United Kingdom 
    (\email{juan.kuntz-nussio@warwick.ac.uk}).}
  \and
  Philipp Thomas\thanks{Department of Mathematics, Imperial College London, London SW7 2AZ, United Kingdom (\email{p.thomas@imperial.ac.uk}).}
  \and
  Guy-Bart Stan\thanks{Department of Bioengineering, Imperial College London, London SW7 2AZ, United Kingdom
    (\email{g.stan@imperial.ac.uk}).}
  \and
  Mauricio Barahona\thanks{Department of Mathematics, Imperial College London, London SW7 2AZ, United Kingdom
    (\email{m.barahona@imperial.ac.uk}).}
}

\begin{document}
\maketitle
\begin{abstract}
We study a class of countably-infinite-dimensional linear programs (CILPs) 
whose feasible sets are bounded subsets of appropriately defined spaces of measures. {The optimal value, optimal points, and minimal points of these CILPs can
be approximated by solving finite-dimensional linear programs. We show how to construct finite-dimensional programs that lead to approximations with easy-to-evaluate error bounds, and we prove that the errors converge to zero as the size of the finite-dimensional programs approaches that of the original problem.} 
We discuss the use of our methods in the computation of the stationary distributions, occupation measures, and exit distributions of Markov~chains.
%
\end{abstract}




\section{Introduction}\label{sec:intro}

A countably infinite linear program (CILP) is a linear program (LP) with countably many decision variables and constraints.
 CILPs arise in a variety of applications, e.g.\ network flow problems~\cite{Romeijn2006,Sharkey2008,Ryan2018}, production planning~\cite{Smith1998,Huang2005}, equipment replacement and capacity expansion~\cite{Jones1988,Huang2005}, semi-infinite linear programs~\cite{Anderson1987,Lerma1998}, search problems in robotics~\cite{Demaine2006}, robust optimisation~\cite{Ghate2015}, and, prominently, in optimal control problems tied to Markov chains~\cite{Puterman1994,Hernandez-Lerma1999,Altman1999,Feinberg2002,Guo2011}. 
  
Here, we consider a class of CILPs stemming from the analysis of Markov chains with countably infinite state spaces.
Certain measures {$\rho$} associated with Markov chains (e.g.\ the stationary distributions and occupation measures) are feasible points of CILPs, and the optimal values of these CILPs bound {integrals of the measures of the form $\rho(f):=\int f \, d\rho$}, where $f$ denotes a $\rho$-integrable function.
Typically, the CILPs cannot be solved directly, yet their feasible sets often lie inside a bounded space of measures
{(see Remark~\ref{rem:boundedspace})}, a fact that can be made explicit by finding moment bounds satisfied by the measures {(i.e.\ bounds on $\rho(w)$ for norm-like functions $w$, as in~\eqref{eq:wfunc0}--\eqref{eq:wfunc})}.
%
Such bounds can be obtained using either Foster-Lyapunov criteria~\cite{Meyn1993b,Meyn2009,Glynn2008,Kuntzthe,Kuntz2019a} or mathematical programming~\cite{Hernandez-Lerma2003,Schwerer1996,Helmes2001,Kuntzthe,Kuntz2019,Sakurai2017,Dowdy2017,Dowdy2018,Ghusinga2017a}.

In this paper, we derive finite-dimensional LPs that approximate a given CILP by using a moment bound to truncate the (infinite) set of decision variables and replace its (infinite) set of constraints with finitely many.
Based on the ensuing finite-dimensional LPs, we introduce two approximation schemes: 
\begin{itemize}
\item Scheme A yields lower and upper bounds on the optimal value of the CILP and approximations of its optimal points; 
\item Scheme B yields approximations of the minimal point (i.e.\ a feasible point that is element-wise no greater than every other feasible point), if it exists, 
with an easily computable bound on the error. 
\end{itemize}
Both schemes are shown to converge 
as the size of the approximating LP approaches that of the original CILP {(i.e.\ as $r$ in~\eqref{eq:lpap} approaches infinity).}
The moment bound also allows us to quantify the error of the approximations.

{ The idea of approximating infinite-dimensional LPs with finite-dimensional ones is not new: it traces back to~\cite{Vershik1968}, if not earlier, and has been extensively explored (see, e.g.\  \cite{Altman1994,Altman1999,Ghate2010,Guo2014,Grinold1977} for CILPs, and also \cite{Puterman1994,Prieto2009,Cavazos-Cadena1986,Prieto2012} for similar approximation techniques involving dynamic programming). The main novelty of our work lies in the practical error bounds that accompany our schemes. 
In this paper, we have three aims: (i) to present the schemes and their error bounds in a general setting, clarifying the ingredients necessary for their design (see Section~\ref{conclu}); (ii) to provide the proofs and technical details pertaining to their convergence 
(see, in particular, Theorems~\ref{solvable},~\ref{convlpt},~and~\ref{lconvlpt} and Corollaries~\ref{cor:optimal}~and~\ref{uniqueth4}); 
and (iii) to discuss the use of these methods in the exit problem for Markov chains (Sections~\ref{applications}~and~\ref{applicationsrev}).
The technical proofs and theorems presented in this paper provide the theoretical foundation for the applications presented in~\cite{Kuntz2019}, where we demonstrated without proofs the practical use of the schemes in a particular setting of interest (i.e.\ the approximation of stationary distributions of stochastic reaction networks). 
%

We remark that our schemes are not reliant on the Lasserre hierarchy~\cite{Laurent2009,Lasserre2009}. Specifically:
the schemes do not use moment and localising matrices to constrain their feasible sets; the problem data need not be polynomial (or 
more generally, have any algebraic structure); and the convergence proofs do not require \textit{positivstellens\"atze}. 
If the problem data is polynomial, then our schemes can use moment bounds obtained with the Lasserre hierarchy to: (i) tackle problems beyond the reach of the Lasserre hierarchy
(e.g.\ to produce bounds 
for non-polynomial test functions); 
(ii) recover converging bounds in cases where the hierarchy has no convergence guarantees and its progress stalls (see \cite[Section V.B]{Kuntz2019} for an example). 
In this sense, the schemes proposed here enlarge the already impressive list of applications of the Lasserre hierarchy.}

The paper is organised as follows. In Section~\ref{thelp}, we introduce the class of CILPs we consider, and we motivate our work with two problems taken from the Markov chain literature. In Section~\ref{sec:A}, we present approximations of the CILPs' optimal values and points, prove their convergence, and obtain
practically useful 
error bounds for the optimal values. 
In Section~\ref{sec:B}, we derive approximations of the  CILPs' minimal points and show how to bound their error. 
In Section~\ref{sec:schemes}, we present two numerical procedures (Schemes A and B) to compute these approximations. Afterwards, in Section~\ref{applicationsrev}, we describe their application to the motivating problems of Section~\ref{thelp}. We conclude the paper with a discussion in Section~\ref{conclu}. { The paper has three appendices: Appendix~\ref{appendix} contains the more technical proofs; Appendix~\ref{wtv} clarifies the relationship between the various notions of convergence we use; Appendix~\ref{shedvarfin} explains how to dispense of a technical assumption in our results, which is only assumed to simplify the exposition.}  

\section{The linear program and two motivating problems}\label{thelp} 
\subsection{Statement of the problem} 
Let $\x,\y$ denote two countable (possibly infinite) sets. In various settings, we are interested in  linear programs where the solutions $\rho=(\rho(x))_{x \in \x}$ lie in $\ell^1$, the space of absolutely summable sequences indexed by $\x$,
$$\ell^1:=\left\{ \left(\rho(x)\right)_{x \in \x}:\sum_{x\in\x}\mmag{\rho(x)}<\infty\right\},$$
and satisfy the following properties:
\begin{enumerate}[(i)]
\item $\rho$ is non-negative: 
\begin{equation}
\label{eq:positive_measure}
\rho(x) \geq 0 \quad \forall x \in \x;
\end{equation}
\item $\rho$ satisfies a system of linear equations
\begin{equation}
\label{eq:lineqs1}
\rho H(x) :=\sum_{x'\in\x}\rho(x') h(x',x) = \phi(x) \quad\forall x\in\x,\end{equation}
where $H:=(h(x',x))_{x,x'\in\x}$ is a given Metzler matrix {(c.f. Assumption~\ref{assumption}$(i)$)} 
and $\phi:=(\phi(x))_{x\in\x}$ a given measure in $\ell^1$; 
\item the image of $\rho$ under a given non-negative matrix $G:=(g(x,y))_{x\in\x,y\in\y}$, 
\begin{equation}
\label{eq:psi}
\psi(y):=\rho G(y)=\sum_{x\in\x} \rho(x)g(x,y)\quad \forall y\in\y, 
\end{equation}
is a probability distribution; that is, $\psi$ is non-negative and has unit mass:  
\begin{align}
\label{eq:gsums1}
&\rho(g):=   \sum_{x\in\x}g(x)\rho(x) = \sum_{y\in\y} \psi(y) = 1, \\ 
\text{where}& \qquad  g(x) :=\sum_{y\in\y}g(x,y) \quad \forall x\in\x. \nonumber
\end{align}
Here, $\rho(g)$ denotes the $\rho$-integral of $g$.
\end{enumerate}
If $\x$ and $\y$ are infinite (or very large), linear programs constrained by~\eqref{eq:positive_measure}--\eqref{eq:gsums1} cannot be solved directly.

In many important cases, it is possible to find a 
real-valued function $w$ that fulfills two properties:
\begin{enumerate}[(i)]
\item $w$ is \emph{norm-like}:
\begin{itemize}  
\item it is non-negative,  
\begin{equation}\label{eq:wfunc0}
w(x) \geq 0 \quad \forall x \in \x,
\end{equation}
\item and it has finite sublevel sets, 
\begin{equation}
\label{eq:wfunc}
\x_r:=\{x\in\x:w(x)<r\} \quad \forall r\in\zp;
\end{equation}
\end{itemize}
\item the $\rho$-integral of $w$ fulfills a so-called \textit{moment bound}: 
\begin{equation}
\label{eq:ggmomb}
\rho(w)=\sum_{x\in\x}w(x)\rho(x)\leq c,
\end{equation} 
where $c$ is a constant.
Bounds of this type
can be obtained through various analytical
and numerical  
methods. 
For instance, the moment bounds for the applications in this paper
can be computed using either Foster-Lyapunov criteria~\cite{Glynn2008,Meyn2009,Meyn1993b,Kuntzthe,Kuntz2019a} or mathematical programming approaches, i.e.\ by solving linear or semidefinite programs~ \cite{Hernandez-Lerma2003,Helmes2001,Kuntzthe,Schwerer1996,Kuntz2019,Sakurai2017,Dowdy2017,Dowdy2018,Ghusinga2017a}.
\end{enumerate}
When such a function $w$ and constant $c$ are found,
the solutions $\rho$ 
belong to a set constrained by linear equalities and inequalities:
\begin{equation}
\label{eq:l2}
\L:=\left\{\rho\in\ell^1: \begin{array}{l} \rho H(x)=\phi(x)\quad\forall x\in\x,\\ \rho (g)=1,\\ \rho(w)\leq c,\\\rho(x) \geq 0 \quad \forall x \in \x.
\end{array}\right\}{.} 
\end{equation}
%
%
%
Our aim is to solve linear programs {whose feasible set is} $\mathcal{L}$. Specifically, our problems of interest can be stated as follows.
\begin{problem}[CILPs over measures satisfying moment bounds] 
\label{prob:CILP}
We wish to optimise the $\rho$-integral of a given real-valued function $f:\x \to \r$ over the set of measures $\mathcal{L}$, i.e.\ to compute the optimal values $l_f, u_f $ in
\begin{align}
\label{eq:boundconvthe1}
& l_f := \inf\{\rho(f):  \rho\in\L\},\\
\label{eq:boundconvthe2}
& u_f := \sup\{\rho(f):  \rho\in\L\}.
\end{align}
Should they exist, we are also interested in finding optimal points $\rho_*,\rho^*\in\cal{L}$ achieving the optimal values: 
$$\rho_*(f)=l_f \quad \text{and} \quad  \rho^*(f)=u_f,$$ 
as well as the minimal point $\rho_m\in\cal{L}$ (should it exist): the unique feasible point that is element-wise no greater than all feasible points,
\begin{equation}
\label{eq:minimal}
\rho_m(x)\leq \rho(x)\quad\forall x\in\x, \enskip \forall \rho\in\L.
\end{equation}
\end{problem}

Here we study approximations of the  
CILPs~\eqref{eq:boundconvthe1}--\eqref{eq:boundconvthe2} 
under the following general assumption.
\begin{assumption}\label{assumption}
Consider a CILP over the set $\L$~\eqref{eq:l2} defined by the real matrices $H:=(h(x',x))_{x',x\in\x}$ and $G:=(g(x,y))_{x\in\x,y\in\y}$, and the measure $\phi\in\ell^1$.
We assume that we have available a 
norm-like function $w:\x \to [0,\infty)$ such that:
\begin{enumerate}[label=(\roman*)]
\item $H$ is Metzler:
$h(x',x)\geq 0,\enskip\forall x'\neq x \in \x.$
\item $G$ is nonnegative:
$g(x,y)\geq0, \enskip \forall x\in\x, \enskip \forall y\in\y$.
\item 
There exist known positive constants $a_1,a_2,\dots$ such that
\begin{equation}
\label{eq:gbounds}
\sup_{x\not\in\x_r}\frac{g(x)}{w(x)}\leq a_r\quad\forall r\in\zp,\qquad\lim_{r\to\infty}a_r=0,
\end{equation}
where $g(x)$ is $x$-row sum~\eqref{eq:gsums1} of $G$ and $\x_r$ is the $r$-sublevel set~\eqref{eq:ggmomb} of $w$.
\item For each $x\in\x$, at least one of the following three conditions is fulfilled:
either $w(x)>0$;
or $g(x)>0$;
or there exists $x_1,\dots,x_l\in\x$ such that $h(x,x_1)h(x_1,x_2)\dots h(x_{l-1},x_l)>0$ and $g(x_l)>0$.
\item 
The support of each of the columns of 
$H$ is finite:
\[
\supp{h(\cdot,x)}:= \{x'\in\x:h(x',x)\neq0\}\enskip\text{is finite,}\quad \forall  x \in \x. \]
\item $\L$ is non-empty.
\end{enumerate}
\end{assumption}
Assumptions~$(i,ii)$ ensure that $\rho H(x)$ and $\rho G(y)$ are absolutely convergent or $+\infty$ for any non-negative $\rho\in\ell^1$. 
Assumption~$(iii)$ allows us to control the feasible points outside of the sublevel sets of $w$ and is key in deriving our finite-dimensional approximating LPs. Assumption~$(iv)$ guarantees that the entries $\rho(x)$ 
are bounded---a requirement for the proofs of Sections~\ref{sec:A}--\ref{sec:B}. 
The particular statement of Assumption~$(iv)$ is motivated by the applications in Section~\ref{applications} for which it is natural, {(e.g.\ for the exit problem it asks that the chain may leave the domain from every state $x$ for which $w(x)=0$)}.
Assumption~$(v)$ is trivially satisfied in many applications~\cite{Sharkey2008,Ghate2010,Ghate2013} (for Markov chains, it requires that each state is reachable in a single jump from at most a finite number of states), and simplifies 
the derivation of some of our results. {It} can, however, be circumvented by increasing the size of the approximating LPs 
(see Appendix~\ref{shedvarfin}). {Lastly, Assumption~$(vi)$ ensures that the CILPs~\eqref{eq:boundconvthe1}--\eqref{eq:boundconvthe2} are non-trivial.}

For a broad class of functions $f$, the CILPs~\eqref{eq:boundconvthe1}--\eqref{eq:boundconvthe2} are guaranteed to be well-posed and \emph{solvable}, i.e.\ their optimal values are finite and their optimal points~exist.

\begin{theorem}[The CILPs are solvable]\label{solvable}
Suppose that Assumption \ref{assumption} is satisfied and assume that $f$ belongs to the set $\cal{W}$ of real-valued functions on $\x$ that eventually grow strictly slower than the norm-like function $w$:
\begin{equation}\label{eq:grs}\cal{W}:=\left\{f:\lim_{r\to\infty}\left(\sup_{x\not\in \x_r}\frac{\mmag{f(x)}}{w(x)}\right)=0\right\}.\end{equation}
Then, we have that:
\begin{enumerate}[(i)]
\item the sum $\rho(f)$ is absolutely convergent, $\forall \rho\in\cal{L}$, 
\item the optimal value $l_f$ of \eqref{eq:boundconvthe1} is finite, 
\item there exists at least one optimal point $\rho_*\in \cal{L}$ satisfying $\rho_*(f)=l_f$.
\end{enumerate}
Note that by replacing $f$ by $-f$, the above holds identically for optimal value $u_f$~\eqref{eq:boundconvthe2} and corresponding optimal points $\rho^*$.
\end{theorem}
\begin{proof}See Appendix~\ref{appendix}.\end{proof}

As explained in the following remark, Assumption~\ref{assumption} ensures that the feasible sets of CILPs~\eqref{eq:boundconvthe1}--\eqref{eq:boundconvthe2} are bounded subsets of an appropriate normed space, a fact that underlies the results of this paper. 
\begin{remark}[The feasible set is contained in a bounded measure space]\label{rem:boundedspace} {If $\cal{N}$ denotes the null set $\{x\in\cal{X}:w(x)=0\}$ of the norm-like function $w$, then
$$\tilde{w}(x):=\left\{\begin{array}{cl}1&x\in\cal{N}\\w(x)&x\not\in\cal{N}\end{array}\right.\quad\forall x\in\cal{X}$$
defines the norm
$$\norm{{\rho}}_{\tilde{w}}:=\sum_{x\in\x}\mmag{\rho(x)}\tilde{w}(x)$$
on the weighted space $\{\rho\in\ell^1:\norm{\rho}_{\tilde{w}}<\infty\}$. For any $x$ in the null set $\cal{N}$, Assumption~\ref{assumption}$(i,ii,iv)$ and the constraints $\rho H=\phi $ and $\rho(g)=1$ in \eqref{eq:l2} imply that $\rho(x)\leq c_x$ uniformly over $\rho$ in $\cal{L}$ for some constant $c_x<\infty$. Because $w$ is norm-like, $\cal{N}$ is finite and it follows that $c_{\cal{N}}:=\sum_{x\in\cal{N}}c_x<\infty$. Consequently, the constraints $\rho(w)\leq c$ and $\rho \geq0$   in \eqref{eq:l2} imply that the feasible set $\cal{L}$ lies inside the bounded normed space $(\mathbb{B}_{c+c_\cal{N}},\norm{\cdot}_{\tilde{w}})$, where $\mathbb{B}_{c+c_\cal{N}}:=\{\rho\in\ell^1:\norm{\rho}_{\tilde{w}}\leq c+c_\cal{N}\}$. Therefore we say that \eqref{eq:boundconvthe1}--\eqref{eq:boundconvthe2} are CILPs over \textit{bounded measure spaces}.}
\end{remark}

\subsection*{Notation}
Throughout, { we denote the set of positive integers $\{1,2,3,\dots\}$ by $\zp$ and} we use the conventions
$\sup \emptyset  = - \infty$ and $\inf \emptyset = + \infty$.
A countable set $\x$ has cardinality $\mmag{\x}$
and power set $\twx:=\{A:A\subseteq\x\}$.
We abuse notation by using $\rho$ to denote both a measure on $(\x,\twx)$ and its density ($\rho(x):=\rho(\{x\})$~$\forall x\in\x$) so that
\begin{equation}
\label{eq:not1}
\rho(A)=\sum_{x\in A}\rho(x)\quad\forall A\subseteq \x.
\end{equation}
Given the above, we identify the space of finite, signed measures on $(\x,2^\x)$ with $\ell^1$. 

\subsection{Two motivating problems from Markov chain theory:
stationary distributions and exit problems}\label{applications} 
Our work is motivated by the analysis of Markov processes on countably infinite state spaces (a.k.a.\ Markov chains, or \emph{chains} for short).
Specifically, we are interested in computing the stationary distributions of a chain and the exit distributions and occupation measures associated with its exit times.
We now show briefly how these two problems can be mapped to Problem~\ref{prob:CILP} for both discrete-time and continuous-time chains.

\subsubsection{The discrete-time case}
Let $(X_n)_{n\in\n}$ denote a time-homogeneous discrete-time Markov chain taking values in a countable state space $\s$, with one-step matrix $P:=(p(x,y))_{x,y\in\s}$ and initial distribution  $\gamma:=(\gamma(x))_{x\in\s}$:
$$p(x,y)=\Pbl{\{X_1=y\}|\{X_0=x\}},\quad \gamma(x)=\Pbl{\{X_0=x\}},\quad\forall x,y\in\s,$$
where $\mathbb{P}_\gamma$ denotes the underlying probability measure (and the subscript emphasises that the initial state is sampled from $\gamma$).

\paragraph{Stationary distributions} A probability measure $\pi$ on $\s$ is a \emph{stationary distribution} of the chain if sampling its initial state from $\pi$ ensures that the chain remains distributed according to $\pi$ for all time:
$$\Pbp{\{X_n=x\}}=\pi(x)\quad\forall x\in\s,\enskip n\in\n.$$
As the following well-known corollary shows, the stationary distributions (satisfying the moment bound) are the feasible points of a CILP of the type in Problem~\ref{prob:CILP}. The optimal points are those maximising or minimising an average of interest over the set of stationary distributions and the optimal value is said maximum or minimum.
\begin{corollary}[{\cite[Theorem A.I.3.1]{Asmussen2003}}]\label{cor:dtstat}
The set of stationary distributions $\pi$ that satisfy the moment bound $\pi(w)\leq c$ is the set $\L$~\eqref{eq:l2} with
\begin{equation}
\label{eq:dtst}
\begin{array}{l} 
\x:=\s,\quad\y:=\s,\\
h(x',x):=p(x',x)-1_{x'}(x)\quad\forall x',x\in\s,\\
\phi(x):=0\quad\forall x\in\s,\\ 
g(x,y):=1_x(y)\quad\forall x,y\in\s,
\end{array}
\end{equation}
where $1_{x'}$ denotes the indicator function of state $x'$.
\end{corollary}
\paragraph{Exit distribution and occupation measure}
The \emph{exit time} $\sigma$ of the chain $(X_n)_{n\in\n}$ from a subset $\cal{D} \subseteq \s$ (the \emph{domain}) is the first time that the chain lies outside~$\cal{D}$:
$$\sigma:=\inf{\{n\in\n:X_n\not\in\cal{D}\}},$$
with $\sigma: = \infty$ if the chain never leaves $\mathcal{D}$.
Associated with the exit time, there is an \emph{exit distribution} $\mu$ and \emph{occupation measure} $\nu$:
\begin{align*}
\mu(x):= & \Pbl{\{X_\sigma=x,\sigma<\infty\}}\quad   \forall x\not\in\cal{D}, \\
\nu(x):= & \Ebl{\sum_{m=0}^{\sigma-1} 1_{x}(X_m)}\qquad\quad \forall x\in\cal{D}, 
\end{align*}
i.e.\ $\mu(x)$ denotes the probability that the chain exits~$\cal{D}$ via state $x$, and $\nu(x)$ denotes the expected number of visits that the chain makes to $x$ before exiting~$\cal{D}$. The occupation measure is the minimal point of a CILP of the type in Problem~\ref{prob:CILP} and the exit distribution is its image through an appropiately defined $G$:
\begin{corollary}[{{\cite[Theorem 11.7]{Kuntz2020}}}]\label{cor:exit1}
If $\sigma<\infty$ almost surely (i.e.\ chain eventually leaves $\cal{D}$ with probability one),  $\gamma(\cal{D})=1$ (i.e.\ the chain starts inside $\cal{D}$), and the occupation measure $\nu$ satisfies a moment bound $\nu(w)\leq c$, then $\nu$ is the minimal point~\eqref{eq:minimal} of the set $\L$ in~\eqref{eq:l2} with
\begin{equation}\label{eq:dtex}\begin{array}{l}\x:=\cal{D},\quad\y:=\cal{D}^c \\ 
h(x',x):=p(x',x)-1_{x'}(x)\quad\forall x',x\in\cal{D},\\
 \phi(x):=-\gamma(x)\quad\forall x\in\cal{D}, \\ g(x,y):=p(x,y)\quad \forall x\in\cal{D},\enskip \forall y\not\in\cal{D},\\
\end{array}\end{equation}
where $\cal{D}^c = \s \backslash \cal{D}$ denotes the complement of $\cal{D}$ in $\s$.  
Furthermore, the exit distribution is given by $\mu= \nu G$. 
\end{corollary}

\subsubsection{The continuous-time case} 
Let $(X_t)_{t\geq 0}$ denote a minimal time-homogeneous continuous-time Markov chain {(e.g.\ see \cite[Sections~26,~37]{Kuntz2020})} taking values in a countable state space $\s$ with rate matrix $Q:=(q(x,y))_{x,y\in\s}$ and initial distribution $\gamma:=(\gamma(x))_{x\in\s}$:
$$q(x,y)=\lim_{t\to0}\frac{\Pbl{\{X_t=y,t<T_\infty\}|\{X_0=x\}}}{t},\quad \gamma(x)=\Pbl{\{X_0=x\}},\enskip\forall x,y\in\s,$$
where $T_\infty$ denotes the explosion time of the chain and $\mathbb{P}_\gamma$ is the underlying probability measure. We assume that $Q$ is stable and conservative, i.e.\
\begin{equation}
\label{eq:stabcon} 
-q(x,x)=\sum_{\substack{y \in \s \\ y\neq x}} q(x,y)<\infty\qquad\forall x\in\s.
\end{equation}

\paragraph{Stationary distribution}  A probability measure $\pi$ on $\s$ is a stationary distribution of the continuous-time chain if 
$$\Pbp{\{X_t=x,t<T_\infty\}}=\pi(x)\quad\forall x\in\s,\enskip t\geq0.$$
As for the discrete-time case, the   stationary distributions satisfying the moment bound are the feasible points of CILPs of the type in Problem~\ref{prob:CILP}:
\begin{corollary}[{{\cite[Theorem~1]{Miller1963}}}]
If $Q$ is regular (i.e.\ 
$\Pbl{\{T_\infty=\infty\}}=1$ for all probability 
distributions $\gamma$), the set of stationary distributions that satisfy the moment bound $\pi(w)\leq c$ is the set $\L$ in~\eqref{eq:l2} with 
\begin{equation}
\label{eq:ctst}
\begin{array}{l}
\x:=\s,\quad\y:=\s,\\ 
h(x',x):=q(x',x)\quad\forall x',x\in\s,\\
\phi(x):=0\quad\forall x\in\s,\\ 
g(x,y):=1_x(y)\quad\forall x,y\in\s.
\end{array}
\end{equation}
\end{corollary}

\paragraph{Exit distribution and occupation measure}
The exit time $\tau$ of $(X_t)_{t\geq0}$ from a domain $\cal{D}$ is the first time that the chain lies outside of $\cal{D}$:
$$\tau=\inf\{0\leq t<T_\infty: X_t\not\in \cal{D}\},$$
with $\tau: = \infty$ if the chain never leaves $\mathcal{D}$. 
Associated with $\tau$, there is an \emph{exit distribution} $\mu$ and \emph{occupation measure} $\nu$:
\begin{align*}
\mu(x):= & \Pbl{\{X_\tau=x,\tau<\infty\}} \quad \forall x\not\in\cal{D}, \\ \nu(x):= & \Ebl{\int_0^{\min\{\tau,T_\infty\}} 1_{x}(X_t)dt} \quad \forall x\in\cal{D},
\end{align*}
which characterise where the chain exits the domain ($\mu$), and where inside the domain the chain spends its time until it exits ($\nu$). Similarly as in the discrete-time case, the occupation measure is the minimal point of a CILP of the type in Problem~\ref{prob:CILP} and the exit distribution is its image through an appropriately defined $G$:
\begin{corollary}[{{\cite[Theorem 36.11]{Kuntz2020}}}] \label{cor:exit2}
If $\tau < \infty$ almost surely, 
$\gamma(\cal{D})=1$, 
and $\nu(w)\leq c$,  the occupation measure is the minimal point~\eqref{eq:minimal} of the set $\L$ in~\eqref{eq:l2} with
\begin{equation}\label{eq:ctex}
\begin{array}{l}\x:=\cal{D},\quad\y:=\cal{D}^c,\\ h(x',x):=q(x',x)\quad\forall x',x\in\cal{D},
\\
 \phi(x):=-\gamma(x)\quad\forall x\in\cal{D},\\ g(x,y):=q(x,y)\quad \forall x\in\cal{D},\enskip \forall y\not\in\cal{D},\\
 \end{array}\end{equation}
Furthermore, the exit distribution is given by $\mu= \nu G$. 
\end{corollary}

\section{Bounding the optimal values and approximating the optimal points of the CILP}\label{sec:A} 

Consider Problem~\ref{prob:CILP} under Assumption~\ref{assumption}. 
To derive finite-dimensional approximations of CILPs \eqref{eq:boundconvthe1}--\eqref{eq:boundconvthe2}, we start by truncating the set $\x$ using the sublevel sets $\x_r$ in~\eqref{eq:wfunc} of the norm-like function $w$. We then define the restriction $\rho_{|r}$ to $\x_r$ of each 
feasible point~$\rho$~in~$\L$,
\begin{equation}
\label{eq:restr}
\rho_{|r}(x):=\left\{\begin{array}{ll}\rho(x)&\text{if }x\in \x_r\\ 0 &\text{if }x\not\in \x_r\end{array}\right.\quad\forall x\in \cal{X},
\end{equation} 
and the following set of measures $\rho^r$: 
\begin{equation}
\label{eq:lpap}
\L_r:=\left\{\rho^r\in\ell^1: 
\begin{array}{l} 
    \rho^r H(x)=\phi(x) \quad\forall x\in\e_r,\\
    1-ca_r\leq\rho^r(g)\leq 1, \\
    \rho^r(w)\leq c,\\ 
    \rho^r(x)\geq0 \quad\forall x\in\x,
    \\ \rho^r(\x_r^c)=0. 
\end{array}
\right\} \quad \forall r \in \zp.
\end{equation}
Here $\x_r^c = \x \backslash \x_r$ denotes the complement of $\x_r$ in $\x$; the positive constant $a_r$ is as in~\eqref{eq:gbounds}; $c$ denotes the moment bound constant in~\eqref{eq:ggmomb};  
and
\begin{equation}\label{eq:er}\e_r:=\{x\in\x_r:h(x',x)=0 \enskip\forall x'\not\in\x_r\}.\end{equation}
The set $\L_r$ is an \emph{outer approximation} of $\L$ in the sense that the restriction of every feasible point in $\L$ belongs to $\L_r$. 
\begin{lemma}[The outer approximation property of $\L_r$]\label{thebounds} Suppose that Assumption~\ref{assumption} is satisfied. For any $r$~in~$\zp$, 
$\L_r$ is an outer approximation of $\L$:
$$\rho_{|r} \in \L_r \quad \forall \rho \in \L.$$
\end{lemma}

\begin{proof} 
Let $\rho $ be any feasible point in $\cal{L}$ and $\rho_{|r}$ be its restriction~\eqref{eq:restr} to $\x_r$.
If $x$ belongs to $\e_r$, $\rho H(x)=\phi(x)$ only involves entries $\rho(x')$ 
indexed by $x'$ in $\x_r$, and so
\begin{equation}\label{eq:eqnstrunc}
\rho_{|r}H(x)=\rho H(x)=\phi(x) \quad\forall x\in\e_r.
\end{equation} 
Because $w$ is a non-negative function, it follows from the definition~\eqref{eq:l2} that
\begin{equation}\label{eq:rhormomb}\rho_{|r}(w)\leq\rho(w)\leq c,\quad\rho_{|r}(x)\geq 0\quad\forall x\in\x.\end{equation}
Next, consider the following generalisation of Markov's inequality: for any nonnegative function $f$ on $\x$,
\begin{align}
\label{eq:markovthe1}
\sum_{x\not\in\x_r}f(x)\rho(x)&\leq \left(\sup_{x\not\in\x_r}\frac{f(x)}{w(x)}\right)\sum_{x\not\in\x_r}w(x)\rho(x)
\leq\left(\sup_{x\not\in\x_r}\frac{f(x)}{w(x)}\right)\rho(w)\\&\leq c\left(\sup_{x\not\in\x_r}\frac{f(x)}{w(x)}\right)\nonumber.
\end{align}
Setting $f$ to be the row-sum function $g$  in~\eqref{eq:gsums1}, we find that
\begin{equation}
\label{eq:gconstrel}
1-ca_r\leq  1-c\left(\sup_{x\not\in\x_r}\frac{g(x)}{w(x)}\right)\leq 1-\sum_{x\not\in\x_r}g(x)\rho(x)=\rho_{|r} (g)\leq\rho (g)=1,
\end{equation}
where the first inequality follows from Assumption~\ref{assumption}$(iii)$, the second from~\eqref{eq:markovthe1}, and the remainder from~$\rho(g)=1$ in~\eqref{eq:l2}. Combining \eqref{eq:eqnstrunc}--\eqref{eq:gconstrel}, we have that~$\rho_{|r}$ belongs to $\L_r$.
\end{proof}

The outer approximation property of $\L_r$ has an important consequence:
the optimal values of the CILPs~\eqref{eq:boundconvthe1}--\eqref{eq:boundconvthe2} can be bounded by optimising over $\L_r$, i.e.\ by solving the finite-dimensional LPs
\begin{align}
l^r_f&:=\inf\{\rho^r(f):\rho^r\in\L_r\},\label{eq:favet1}\\
u^r_f&:=\sup\{\rho^r(f):\rho^r\in\L_r\}.  \label{eq:favet2}
\end{align}
Note that these LPs are solvable because they entail optimising a linear function over a compact subset of $\mathbb{R}^{\mmag{\x_r}}$. Lemma~\ref{thebounds} then yields the following bounds:
\begin{corollary}[Bounding the optimal values]\label{thebounds2} Suppose that Assumption~\ref{assumption} is satisfied and let $f$ be any function in $\cal{W}$. 
If $l_f$ and $u_f$ are the optimal values of the CILPs \eqref{eq:boundconvthe1}--\eqref{eq:boundconvthe2} and $l^r_f$ and $u^r_f$ are those of \eqref{eq:favet1}--\eqref{eq:favet2}, then we have that
\begin{equation}
\label{eq:kkk}
l^r_f-c\left(\sup_{x\not\in\x_r}\frac{\mmag{f(x)}}{w(x)}\right)
\leq l_f\leq \rho(f)\leq u_f \leq u^r_f+c\left(\sup_{x\not\in\x_r}\frac{\mmag{f(x)}}{w(x)}\right), 
\enskip \forall \rho\in\cal{L}, \enskip  r \in \zp.
\end{equation}
 Under the following additional assumptions, the bounds~\eqref{eq:kkk} can be sharpened:
\begin{align}
\label{eq:kkk1}
& \text{If $f(x)\geq 0, \, \forall x\not\in\x_r$, then  } l^r_f\leq l_f\leq \rho(f)\quad\forall \rho\in\L,\enskip  r \in \zp.  \\
\label{eq:kkk2}
&\text{If $f(x)\leq 0, \, \forall x\not\in\x_r$, then  } \rho(f)\leq u_f\leq u^r_f \quad \forall \rho\in\L,\enskip  r\in\zp.
\end{align}
\end{corollary}
\begin{proof} 
From Lemma~\ref{thebounds} and the definitions 
\eqref{eq:favet1}--\eqref{eq:favet2}, we have
\begin{align}
\label{eq:fn78weanfa8w33}
l^r_f  \leq\rho_{|r}(f) \leq u^r_f  \Rightarrow 
l^r_f+\sum_{x\not\in\x_r}f(x)\rho(x)  \leq \rho(f)\leq u^r_f+\sum_{x\not\in\x_r}f(x)\rho(x),
\end{align}
for every feasible point $\rho$ in $\cal{L}$.
The inequalities~\eqref{eq:kkk1}--\eqref{eq:kkk2}  follow immediately because all $\rho$~in~$\L$ are non-negative. For \eqref{eq:kkk},  replace $f$ with $\mmag{f}$ in~\eqref{eq:markovthe1} to obtain
\begin{align}
\label{eq:markovthe2}
\mmag{\sum_{x\not\in\x_r}f(x)\rho(x)}\leq \sum_{x\not\in\x_r}\mmag{f(x)}\rho(x)\leq
c\left(\sup_{x\not\in\x_r}\frac{\mmag{f(x)}}{w(x)}\right).
\end{align}
\end{proof}
Our definition~\eqref{eq:wfunc} implies that the truncations $\x_r$  form an increasing sequence that approaches the entire index set $\x$ as $r$ tends to $\infty$:
$$\x_1\subseteq\x_2\subseteq \ldots \subseteq \x_r \subseteq \x_{r+1} \subseteq  \ldots \quad \text{and} \quad \bigcup_{r=1}^\infty\x_r=\x.$$
In the following, we show that the associated outer approximations $\L_r$ converge to $\L$,
and that the sequences of lower bounds $(l^r_f)_{r\in\zp}$ and upper bounds $(u^r_f)_{r\in\zp}$ obtained by solving \eqref{eq:favet1}--\eqref{eq:favet2} for increasing truncations converge to the optimal values $l_f$ and $u_f$, respectively. The notions of convergence we consider here are as follows:
\begin{definition}[Convergence in weak$^*$]\label{def:weak}{ With $\cal{W}$ as in~\eqref{eq:grs} for any given non-negative $w$, we say that a}
 sequence {$(\rho^r)_{r\in\zp}\subseteq\ell^1$} converges to a point {$\rho\in\ell^1$} in weak$^*$ as $r\to\infty$ if and only if
\begin{equation}
\label{eq:weakstardef}
\lim_{r\to\infty}\rho^r(f)=\rho(f)\quad\forall f\in\cal{W}.
\end{equation}
%
%
\end{definition}

\begin{remark}[Convergence in weak$^\ast$ implies convergence in total variation]  If $w$ is norm-like, weak$^*$ convergence of a sequence  {$(\rho^r)_{r\in\zp}\subseteq\ell^1$} to {$\rho\in\ell^1$} implies convergence in total variation:
$$\lim_{r\to\infty}\norm{\rho-\rho^r}=0,$$
where  $\norm{\cdot}$ denotes the total variation norm:
\begin{equation}\label{eq:tvdef}
\norm{\rho-\rho^r}:=\sup_{A\subseteq\x}\mmag{\rho(A)-\rho^r(A)}.
\end{equation}
See~Appendix~\ref{wtv} for details.
\end{remark}

The following theorem formalises the manner in which the outer approximations $\L_r$ converge to the set $\L$ as $r$ tends to infinity.
\begin{theorem}[The outer approximations $\L_r$ converge to $\L$]\label{convlpt} 
Suppose that Assumption~\ref{assumption} is satisfied. The weak$^\ast$ accumulation points of any given sequence $(\rho^r \in \L_r)_{r\in\zp}$ belong to $\cal{L}$, and every subsequence of $(\rho^r)_{r\in\zp}$ has a weak$^\ast$ convergent subsequence. 
\end{theorem}
\begin{proof} This proof is similar to that of the weak$^\ast$ sequential compactness of $\cal{L}$ in Appendix~\ref{subapp}, hence we only sketch it here. A diagonal argument as in Appendix~\ref{subapp}(a) shows that every subsequence of $(\rho^r)_{r\in\zp}$ has a pointwise convergent subsequence. Applying Fatou's lemma as in Appendix~\ref{subapp}(b), we have that the pointwise accumulation points $\rho^\infty$ of $(\rho^r)_{r\in\zp}$ are non-negative and satisfy $\rho^\infty(w)\leq c$. Applying the generalisation~\eqref{eq:markovthe1} of Markov's inequality as in Appendix~\ref{subapp}(c) shows that $(\rho^r)_{r\in\zp}$ has a weak$^*$ convergent subsequence.
Assumption~\ref{assumption}$(ii,iii,v)$ implies that $g$ and $x' \mapsto h(x',x)$ 
belong to $\cal{W}$ and  
$$\e_1\subseteq\e_2\subseteq\dots,\qquad\lim_{r\to\infty}\e_r=\bigcup_{r=1}^\infty \e_r=\x,$$
whence it follows that the weak$^\ast$ accumulation points of $(\rho^r)_{r\in\zp}$ belong to $\cal{L}$.
\end{proof}

Theorem~\ref{convlpt} has the important consequence that, for sufficiently large $r$, the optimal values and points of the finite-dimensional LPs \eqref{eq:favet1}--\eqref{eq:favet2} are close to those of the infinite-dimensional LPs \eqref{eq:boundconvthe1}--\eqref{eq:boundconvthe2}.  

\begin{corollary}[The optimal values and points converge]\label{cor:optimal} Suppose that Assumption~\ref{assumption} is satisfied and $f$ belongs to $\cal{W}$. 

\begin{enumerate}[label=(\roman*)]
    \item (The optimal values) The sequence  $(l^r_f)_{r\in\zp}$ of optimal values of \eqref{eq:favet1} converges to the optimal value $l_f$  of \eqref{eq:boundconvthe1}:
    $$\lim_{r\to\infty}l^r_f=l_f.$$
\item (The optimal points) All weak$^\ast$ accumulation points of any sequence $(\rho^r_*\in\cal{L}_r)_{r\in\zp}$ of optimal points of \eqref{eq:favet1}
belong to the set of optimal points of \eqref{eq:boundconvthe1}
$$\cal{O}_l:=\{\rho_*\in\cal{L}:\rho_*(f)=l_f\},$$
and every subsequence of $(\rho^r_*)_{r\in\zp}$ has a weak$^\ast$ convergent subsequence. In particular, if \eqref{eq:boundconvthe1} has a unique optimal point $\rho_*$ (i.e.\ $\cal{O}_l=\{\rho_*\}$), then $(\rho^r_*)_{r\in\zp}$ converges in weak$^\ast$ to $\rho_*$. 
\item  (The unique feasible point case) If \eqref{eq:boundconvthe1} has a unique feasible point $\rho$ (i.e.\ $\L= \{\rho\}$),
then any 
sequence $(\rho^r \in \L_r)_{r\in\zp}$ of feasible points of \eqref{eq:favet1} converges in weak$^*$ to $\rho$.
\end{enumerate}
\end{corollary}

\begin{proof}$(i,ii)$ Let $\rho^\infty$ be a weak$^\ast$ accumulation point of $(\rho^r_*)_{r\in\zp}$. Given Theorem~\ref{convlpt}, all we need to show is that all $\rho^\infty \in \cal{O}_l$. Pick any   subsequence $(\rho^{r_k})_{k\in\zp}$ converging to $\rho^\infty$ in weak$^\ast$. Theorem~\ref{convlpt} implies that $\rho^\infty$ belongs to $\cal{L}$, and so $\rho^\infty(f)\geq l_f$. On the other hand, $f$ belongs to $\cal{W}$ and \eqref{eq:kkk} implies that
\begin{align*}\rho^\infty(f)&=\lim_{k\to\infty}l^{r_k}_f=\lim_{k\to\infty}l^{r_k}_f-\lim_{k\to\infty}c\left(\sup_{x\not\in\x_{r_{k}}}\frac{\mmag{f(x)}}{w(x)}\right)\\
&=\lim_{k\to\infty}\left(l^{r_k}_f-c \left( \sup_{x\not\in\x_{r_{k}}}\frac{\mmag{f(x)}}{w(x)}\right) \right)\leq l_f.\end{align*}
Hence, $\rho^\infty(f) = l_f$ and $\rho^\infty\in\cal{O}_l$ for any given accumulation point $\rho^\infty$ of $(\rho^r_*)_{r\in\zp}$. 

$(iii)$ follows immediately from $(ii)$ by setting $f:=0$.
\end{proof}

\begin{remark} \label{rem:up}
Replacing $f$ by $-f$, Corollary~\ref{cor:optimal} holds identically for the optimal values ($u_f$ and $u^r_f$) and optimal points ($\rho^*$ and $\rho^{*r}$) of~\eqref{eq:boundconvthe2}~and~\eqref{eq:favet2}. 
\end{remark}
In some applications, including those discussed in Section~\ref{applicationsrev}, 
the images $\psi = \rho G$~\eqref{eq:psi} of the feasible points $\rho \in \L$ are of interest. To approximate these, we use the images $\psi^r=\rho^r G$ of the feasible points of $\cal{L}_r$: 
\begin{proposition}[Convergence of the images of the optimal points]
\label{uniqueth3} 
Suppose that Assumption~\ref{assumption} is satisfied.
If the sequence of points $(\rho^r\in\L_r)_{r \in \zp}$
converges in weak$^*$ to $\rho \in \L$,  
then the sequence of images $(\psi^r=\rho^r G)_{r \in \zp}$ converges in total variation to $\psi =\rho G$, i.e.\
\begin{equation}
    \lim_{r\to \infty} \norm{\psi^r-\psi} 
    = 0. 
\end{equation}
where the norm is defined in~\eqref{eq:tvdef} 
\end{proposition}
\begin{proof} From the definition~\eqref{eq:tvdef}, we have $\forall r,r'\in\zp$  that
$$\norm{\psi^r-\psi}\leq \sum_{y\in\y}\mmag{\psi^r(y)-\psi(y)}\leq \sum_{x\in\x_{r'}}\mmag{\rho^r(x)-\rho(x)}g(x)+\sum_{x\not\in\x_{r'}}(\rho^r(x)+\rho(x))g(x).$$
Due to Assumption~\ref{assumption}$(iii)$, the rest of the proof is analogous to part (c) in the proof of the weak$^\ast$ sequential compactness of $\cal{L}$, see Appendix~\ref{subapp}.
\end{proof}

In the case of a unique feasible point $\rho$~in~$\cal{L}$, Corollary~\ref{cor:optimal}$(iii)$ and Proposition~\ref{uniqueth3} show that any sequence of feasible points $(\rho^r\in\cal{L}_r)_{r\in\zp}$ of the approximating LPs  converges to $\rho${,} 
and that the images $(\psi^r=\rho^r G)_{r\in\zp}$ converge to the image $\psi=\rho G$. 
Yet $\rho^r$ and $\psi^r$ are uncontrolled approximations of $\rho$ and $\psi$ in the sense that there is no known practical way to compute or bound their errors. 
To remedy this, we show in the following section how to repeatedly apply the results of this section to obtain collections of element-wise lower bounds on $\rho$ and $\psi$ with easily quantifiable errors.

\section{Approximating the minimal point of the CILP}\label{sec:B}

The setup here is identical to that of Section~\ref{sec:A}, only with the additional assumption that $\L$ has a minimal point (i.e.\ a feasible point $\rho_m$ satisfying~\eqref{eq:minimal}). We now focus on approximating $\rho_m$ and its image $\psi_m:=\rho_m G$. 
Given a truncation $\x_r$, let 
\begin{equation}\label{eq:boundingmeasures}l^r(x):=\left\{\begin{array}{ll}l^r_{x}&\text{if }x\in\x_r\\0&\text{if }x\not\in\x_r\end{array}\right.\quad\forall x\in\x,\enskip r\in\zp,\end{equation}
where $l^r_x$ is as in~\eqref{eq:favet1} with $f$ being the indicator function $1_{x}$ of the  state $x$. 

\begin{theorem}[Approximation of the minimal point with computable error bounds]
\label{lconvlpt}
Suppose that Assumption \ref{assumption} is satisfied and  
that $\cal{L}$ has a minimal point $\rho_m$.
Let the approximation $l^r=(l^r(x))_{x\in\x}$ be as in \eqref{eq:boundingmeasures}.
Then, the following hold: 
\begin{enumerate}[label=(\roman*)]
\item The approximation error has an easy-to-compute bound:
\begin{equation}
\label{eq:bound_SchemeB}
\norm{\rho_m-l^r}
\leq u_{\x_r}^r+\frac{c}{r}-l^r(\x_r)=:\Gamma_r 
\quad \forall r \in \zp,
\end{equation}
where $u_{\x_r}^r$ is as in~\eqref{eq:favet2} with $f$ being the indicator function $1_{\x_r}$ of  the truncation~$\x_r$.
\item The approximation $l^r$ converges to $\rho_m$ in weak$^*$ as $r \to \infty$. 
\item In the case of a unique feasible point $\rho$, the error bound $\Gamma_r$ converges to zero: 
$$ 
\lim_{r\to\infty} \Gamma_r  = 0. $$
\end{enumerate}
\end{theorem}

\begin{proof}
$(i)$ From Corollary~\ref{thebounds2}, it follows that $l^r$ bounds $\rho_m$ from~below,
\begin{equation}
\label{eq:llbound}
l^r(x)\leq \rho_m(x)\quad \forall x\in\x,\enskip r\in\zp.
\end{equation}
Hence, $\rho_m-l^r$ is an unsigned measure.
The total variation norm of an unsigned measure is its mass, hence \eqref{eq:bound_SchemeB} follows immediately by setting $f:=1_{\x_r}$ in \eqref{eq:kkk}.

$(ii)$ Note that Corollary~\ref{cor:optimal}$(i)$ shows that 
$l^r$ converges pointwise to $\rho_m$:
\begin{equation}\label{eq:lboundconvpoint}\lim_{r\to\infty}l^r(x)=\rho_m(x)\quad\forall x\in\x.
\end{equation} 
Using Fatou's Lemma and \eqref{eq:llbound}--\eqref{eq:lboundconvpoint} we have
\begin{equation}\label{eq:lwbound}l(w)\leq\lim_{r\to\infty}l^r(w)\leq  \lim_{r\to\infty}\rho_m(w)\leq\rho_m(w)\leq c.\end{equation}
Pick any $f$~in~$\cal{W}$, fix $r,n$~in~$\zp$, and note that
\begin{equation}\label{eq:rhombound}
\mmag{\rho_m(f)-l^r(f)}\leq \sum_{x\in\x_n}(\rho_m(x)-l^r(x))\mmag{f(x)}+\sum_{x\not\in\x_n}(\rho_m(x)+l^r(x))\mmag{f(x)}.
\end{equation}
Given \eqref{eq:lboundconvpoint}--\eqref{eq:rhombound}, the rest of the proof is analogous to part (c) in the proof of the weak$^\ast$ sequential compactness of $\cal{L}$, see Appendix~\ref{subapp}.

$(iii)$ This follows immediately from $(i)$ and Corollary~\ref{cor:optimal}.
\end{proof}

\begin{remark}
The analogous measure $u^r$, composed of upper bounds $u^r_x$ obtained from~\eqref{eq:favet2} with $f:=1_x$, is also an approximation on the feasible points of $\L$, and is also accompanied by an error bound. Although Corollary~\ref{cor:optimal} and Remark~\ref{rem:up} show that $u^r$ converges pointwise, no weak$^*$ convergence can be recovered, see \cite{Kuntzthe} for details. 
\end{remark}

To approximate the image $\psi_m:=\rho_m G$ of the minimal point, we define:
\begin{equation}
    \label{eq:approx_image_minpoint}
l^r_\psi(y):=\left\{\begin{array}{ll}l^r_{g_y}&\text{if }y\in\y_r\\0&\text{if }y\not\in\y_r\end{array}\right.\quad\forall y\in\y,
\end{equation}
where $l^r_{g_y}$ is as in~\eqref{eq:favet1} with $f(\cdot)=g_y(\cdot):=g(\cdot,y)$, and $\y_r$ 
belongs {to} any increasing sequence {of finite
subsets} that approach $\cal{Y}$:
$$\y_1\subseteq\y_2\subseteq \ldots \subseteq \y_r \subseteq \y_{r+1} \subseteq  \ldots \quad \text{and} \quad \bigcup_{r=1}^\infty\y_r=\y.$$
\begin{corollary}[Controlled approximation of the image of the minimal point]
\label{uniqueth4} Suppose that Assumption~\ref{assumption} is satisfied and $\cal{L}$ has a minimal point $\rho_m$ with image $\psi_m=\rho_mG$, and let  $l_\psi^r=(l_\psi^r(y))_{y\in\y}$
be the approximation in \eqref{eq:approx_image_minpoint}. We have that: 
\begin{enumerate}[label=(\roman*)]
\item The approximation error is given by
$$\norm{\psi_m-l^r_\psi}=1-{l^r_\psi(\cal{Y}_r)}.$$
\item The lower bounds $l^r_\psi$ converge in total variation to $\psi_m$:
$$\lim_{r\to\infty}\norm{\psi_m-l^r_\psi}=0.$$
\end{enumerate}
\end{corollary}

\begin{proof} 
$(i)$ Since $G$ is non-negative, Corollary~\ref{thebounds2} implies that
\begin{equation}\label{eq:dfn8aw79nf63awmudsa9}l^r_\psi(y)\leq \psi_m(y)\quad\forall y\in\y,\enskip r\in\zp.\end{equation}
Hence, $\psi_m-l^r_\psi$ is an unsigned measure and its total variation norm is its mass.

$(ii)$  Because $G$ is non-negative, it is easy to see that 
$$\rho G(y) \geq \rho_m G(y)=\psi_m(y) \quad\forall \rho\in\cal{L},\enskip y\in\y,$$
which implies that
$$l_{g_y}:=\inf_{\rho\in\cal{L}}\sum_{x\in\s}\rho(x)g(x,y)\geq \psi_m(y)\quad \forall y\in\y.$$
On the other hand, 
${\psi_m}(y)=\rho_m G(y)$ and $\rho_m$ belongs to $\cal{L}$. Hence,  $$l_{g_y}=\psi_m(y) \quad \forall y\in\y.$$  
Choose any $\varepsilon>0$. Because $\psi_m$ is a probability distribution and $\y_r$ approaches $\y$ as $r\to\infty$, there exists an $r'\in\zp$ such that ${\psi_m}(\y_{r'}^c)\leq \varepsilon/2$. From \eqref{eq:dfn8aw79nf63awmudsa9}, we have that
\begin{align} \norm{\psi_m-l^r_\psi} \leq \sum_{y\in\y_{r'}}(\psi_m(y)-l^r_\psi(y))+2\psi_m(\y_{r'}^c)\leq \sum_{y\in\y_{r'}}(l_{g_y}-l^r_{g_y})+\frac{\varepsilon}{2}.\label{eq:fney8wafhba67gy34w2}
\end{align}
%
Because $\y_{r'}$ is finite and, for each $y$~in~$\y$, the function $x\mapsto g(x,y)$ belongs to $\cal{W}$ (Assumption~\ref{assumption}$(ii,iii)$), Corollary~\ref{cor:optimal}$(i)$ implies that there exists an $R$~in~$\zp$ such~that 
$$\sum_{y\in\y_{r'}}(l_{g_y}-l^r_{g_y})\leq\frac{\varepsilon}{2}\quad\forall r\geq R.$$
Combining the above with~\eqref{eq:fney8wafhba67gy34w2}, we have that
$$\norm{\psi_m-l^r_\psi}\leq \varepsilon  \quad \forall r\geq R.$$ 
Because $\varepsilon$ is arbitrary, the result follows.  
\end{proof}
\section{Approximation schemes}\label{sec:schemes}
The results in Sections~\ref{sec:A}~and~\ref{sec:B} can be used to establish computational procedures to approximate CILPs using finite-dimensional~LPs.

\subsection*{Approximation Scheme A}\label{sec:34} 
The results in Section~\ref{sec:A} establish a procedure to approximate
CILPs~\eqref{eq:boundconvthe1}--\eqref{eq:boundconvthe2} by solving the 
finite-dimensional LPs~\eqref{eq:favet1}--\eqref{eq:favet2}  
defined over outer approximations $\L_r$ associated with truncations $\x_r$.
These finite-dimensional LPs
have $|\x_r|$ decision variables, $|\e_r|\leq|\x_r|$ equality constraints, and $|\x_r|+3$ inequality constraints.
We refer to this general approach as~\emph{Scheme A}, which can be used for 
different purposes:
\begin{itemize}
\item In general, the computed bounds provide converging approximations to the optimal values of the CILP, as shown by Corollary~\ref{cor:optimal}$(i)$. 
Specifically, consider the upper and lower bounds in~\eqref{eq:kkk}:
\begin{equation}
\tilde{l}^r_f:=l^r_f-c\left(\sup_{x\not\in\x_r}\frac{\mmag{f(x)}}{w(x)}\right),\quad \tilde{u}^r_f:=u^r_f+c\left(\sup_{x\not\in\x_r}\frac{\mmag{f(x)}}{w(x)}\right),\quad\forall r\in\zp.
\label{eq:bounds_general}
\end{equation}
The gap
\begin{equation}
\label{eq:approx_error_bounds}
 \gamma_r := \tilde{u}^r_f - \tilde{l}^r_f = u^r_f-l^r_f+2c \left(\sup_{x\not\in\x_r}\frac{\mmag{f(x)}}{w(x)}\right) \quad \forall r \in \zp, 
 \end{equation}
bounds the approximation errors:
\begin{align}
    & \mmag{u_f-\tilde{u}_f^r} \leq \gamma_r \quad \text{and} \quad  \mmag{l_f-\tilde{l}_f^r} \leq \gamma_r \quad\forall r\in\zp. 
\end{align}
\item In other applications, the optimal points $\rho^*$ and $\rho_*$ are of interest.
As shown in Corollary~\ref{cor:optimal}$(ii)$, the optimal points~$\rho^r_*$ and $\rho^{r*}$ of the LPs~\eqref{eq:favet1}--\eqref{eq:favet2} converge to the set of optimal points $\rho_*$ and $\rho^*$ of the CILP~\eqref{eq:boundconvthe1}--\eqref{eq:boundconvthe2}, respectively.

If the optimal point $\rho_*$ is unique, then 
solving the LP~\eqref{eq:favet1} yields convergent approximations $\rho^r_*$ of $\rho_*$ (Corollary~\ref{cor:optimal}$(ii)$), yet with no easy-to-compute quantification of the error. 
If $\rho^*$ is unique, the same applies to $\rho^{r*}$ from~\eqref{eq:favet2}. 
\end{itemize}
In the case of a unique feasible point ($\L=\{\rho\}$), our results strenghten considerably:
\begin{itemize}
\item 
Using the mid-point of $\tilde{l}_f^r$ and $\tilde{u}^r_f$ to approximate the integral $\rho(f)$, we find that the error is bounded above by
\begin{align}
    \label{eq:approx_midpoint}
    & \mmag{\rho(f) - \frac{\tilde{u}_f^r + \tilde{l}_f^r}{2}} \leq \frac{\gamma_r}{2}\quad\forall r\in\zp,
\end{align}
which itself converges to zero: 
$$\lim_{r \to \infty} \gamma_r = 0, \enskip \text{  if  } f \in \cal{W}.$$
Therefore any desired tolerance can be verified by computing bounds from truncations with large enough $r$. 
From the definition~\eqref{eq:approx_error_bounds}, note that
$$\gamma_r\geq 2c\left(\sup_{x\not\in\x_r}\frac{\mmag{f(x)}}{w(x)}\right)\quad\forall r\in\zp.$$
Hence, the truncation parameter $r$ should be chosen large enough that
$$  \sup_{x\not\in\x_r}\frac{\mmag{f(x)}}{w(x)}\leq \frac{\varepsilon}{c},$$
where $\varepsilon$ denotes the desired error tolerance.

\item If the feasible point $\rho$ (or its image $\psi=\rho G$) is of interest, any feasible point $\rho^r$ of $\cal{L}_r$ (or its image $\psi^r=\rho^r G$) provides a convergent approximations to $\rho$ (or to $\psi$), for large enough $r$. However, there is no easy-to-compute quantification of the error of the approximation for $\rho$ (or $\psi$). 
To remedy this, we introduce Scheme B, 
which yields converging approximations of $\rho$ (and $\psi$) with easy-to-compute error bounds.

\end{itemize}

\subsubsection{Approximation Scheme B} 
We refer to \emph{Scheme B} as the process of repeatedly applying Scheme A to compute $\mmag{\x_r}$ pointwise lower bounds $l^r_x$ (or $\mmag{\y_r}$ lower bounds $l^r_{g_y}$)  in order to construct an approximation of the minimal point (or its image) with controlled error given by Theorem~\ref{lconvlpt} (or Corollary~\ref{uniqueth4}).

\section{Motivating problems revisited}\label{applicationsrev}

Schemes A and B are applicable to the motivating problems in Section~\ref{applications}. They yield controlled approximations of: (i) the stationary distributions of a chain, and (ii) the exit distributions and occupation measures associated with exit times of the chain.
First, note that Assumption~\ref{assumption}$(i,ii)$ is satisfied automatically because both one-step matrices $P$ of discrete-time chains and rate matrices $Q$ of continuous-time ones are Metzler.
\subsection*{Stationary distributions} Let $H$ and $G$ be as in \eqref{eq:dtst} for the discrete-time case or \eqref{eq:ctst} for the continuous-time one, and suppose that every stationary distribution $\pi$ satisfies the moment bound $\pi(w)\leq c$.  Because $G$ is the identity matrix, Assumption~\ref{assumption}$(iii,iv)$ is trivially satisfied with $a_r:=r^{-1}$. 

In the case of a unique stationary distribution $\pi$, 
Corollary~\ref{cor:optimal}$(iii)$
yields converging approximations $\rho^r$ 
of $\pi$. 
Furthermore, Corollary~\ref{uniqueth4} yields converging approximations $l^r_\psi$ of $\pi$ with computable error bounds. 

In the non-unique case, there exist several \emph{ergodic distributions} $\pi_i$,
 each with support on a disjoint subset of the state space called a closed communicating class $\cal{C}_i$. 
Because the set of stationary distributions equals that of all convex combinations of the ergodic distributions~\cite[Theorems~17.10,~43.19]{Kuntz2020}, 
%
Corollary~\ref{cor:optimal}$(ii)$ and Remark~\ref{rem:up} show that the optimal points $\rho^{*r}$ of~\eqref{eq:favet2} (with $f:=1_x$ for any state $x$ inside the class $\cal{C}_i$) form converging approximations of $\pi_i$.

The schemes developed here are applied extensively to compute the stationary distributions of continuous-time chains in~\cite{Kuntz2019}.

\subsection*{Exit distributions and occupation measures} Let $H$ and $G$ be as in \eqref{eq:dtex} for the discrete-time case or \eqref{eq:ctex} in the continuous-time case.
For discrete-time chains, 
Assumption~\ref{assumption}$(iii)$ holds with $a_r:=r^{-1}$
because one-step matrices $P$ are row stochastic.
For continuous-time chains, choosing $w(x):=(-q(x,x))^d$ for $d>1$ means that 
\eqref{eq:gbounds} holds with $a_r=r^{1-d}$. 
Corollary~\ref{cor:optimal}$(ii)$ (with $f:=1$) and Proposition~\ref{uniqueth3} 
yield converging approximations $\rho^r_*$ and $\psi^r_*$ of the occupation measure $\nu$ and exit distribution $\mu$, respectively. If error bounds are important, then Theorem \ref{lconvlpt} and Corollary~\ref{uniqueth4} yield the approximations $l^r$ and $l^r_\psi$ accompanied by such bounds. 

For the exit problem, Assumption~\ref{assumption}$(iv)$ is a very mild restriction: it requires that the chain must be able to leave the domain (in one or more jumps) from every state $x$ for which $w(x)=0$. If the exit time is almost surely finite, states from which the chain cannot leave the domain are irrelevant, as the chain has zero probability of visiting them.

\section{Concluding remarks}\label{conclu}

In this paper, we present results on the approximation of countably infinite linear programs defined over bounded measure spaces. The approximations are
 linear programs defined over finite-dimensional subspaces of the original infinite-dimensional space   
(i.e.\ we go from $\ell^1$ to its subspace $\{\rho\in\ell^1:\rho(x)=0, \, \forall x\not\in{\x_r}\}$), 
and are accompanied by bounds on the approximation error incurred. 
The convergence properties of our approximations ensure that
any desired error tolerance may be achieved by picking a large enough truncation $\x_r$.

We presented two approaches. Scheme A yields error-controlled approximations of the optimal values and uncontrolled approximations of the optimal points and their images. Scheme B yields error-controlled approximations of the minimal point (should it exist) and its image. The error control of Scheme B comes at a computational price: it entails solving a potentially large number 
of LPs, instead of one or two as in Scheme A. 
Yet, because each point-wise bound is calculated independently, it is straightforward to parallelise their computation and mitigate this additional computational cost.

In related work, Hern\'andez-Lerma and Lasserre \cite{Lerma1998,Lerma1998b,Hernandez-Lerma1999,Hernandez-Lerma2003} introduced a series of numerical schemes to approximate infinite-dimensional linear programs with  finite-dimensional ones.
Within their setting and notation, our approximation consists of: (i)
\emph{aggregating} the constraints 
$\{\rho H(x) = \phi(x), \enskip \forall x\in\x\}$ 
into 
$\{\rho H(x) = \phi(x), \enskip \forall x\in\x_r\}$; 
and (ii) \emph{relaxing} the constraint 
$\rho(g)=1$ to $1-ca_r\leq \rho(g)\leq 1$. 
Our proof of Theorem~\ref{convlpt} follows analogous ideas to those behind the convergence proofs in~\cite{Lerma1998,Lerma1998b,Hernandez-Lerma1999,Hernandez-Lerma2003}, 
although we require less technical machinery {due to our countable setting.} 

{A notable difference of our work is that 
Hern\'andez-Lerma and Lasserre's approach 
does not require a known moment bound; however, they assume finiteness of the optimal value, which must be verified in practice using the same type of tools (e.g.\  Foster-Lyapunov criteria) used to obtain moment bounds. 
To obtain weak$^\ast$ convergence without an explicit moment bound, Hern\'andez-Lerma and Lasserre fix the objective of their LPs to minimising the integral of a norm-like function. This tacitly guarantees that the optimal points of their approximating LPs satisfy a moment bound. 
Instead, we directly append a moment bound as a constraint to the LP, thus retaining the ability to choose the objective at will without sacrificing the convergence guarantee. By choosing the objective carefully, we are then able to construct approximations with easy-to-compute errors---something that is not possible with the approach in~Refs.~\cite{Lerma1998,Lerma1998b,Hernandez-Lerma1999,Hernandez-Lerma2003}.} The development of numerical  methodologies~\cite{Laurent2009,Lasserre2009,Prajna2002,Henrion2009,Lofberg2004,Zheng2018,Papp2019,Ahmadi2019} that facilitate the computation of moment bounds in the two decades since the original work of
Hern\'andez-Lerma and Lasserre
has made it possible to include such refinements in optimisation frameworks, and has encouraged us to develop the schemes presented in this paper.

The other main difference with the work of Hern\'andez-Lerma and Lasserre is that our results are focused on countable index sets $\x$, whereas the applications in~\cite{Lerma1998,Lerma1998b,Hernandez-Lerma1999,Hernandez-Lerma2003} involve uncountable ones. To obtain finite-dimensional LPs, Hern\'andez-Lerma and Lasserre discretise the index set using a dense countable subset of the space that contains the measures they wish to approximate. 
As a result of this discretisation, the restrictions of the desired measure are not necessarily feasible points of the finite-dimensional LPs. Hence the schemes in~\cite{Lerma1998,Lerma1998b,Hernandez-Lerma1999,Hernandez-Lerma2003} yield only converging \emph{approximations} instead of converging \emph{bounds}. For a more detailed comparison, see \cite[Section~5.5]{Kuntzthe}.

Finally, we remark that although we focused here on a particular type of LPs {motivated by the problems discussed in Section~\ref{applications}, variations of} our approximation schemes can be applied to other CILPs. The critical ingredients required to guarantee convergence of Schemes A and B are: a moment bound, and boundedness of the entries of the feasible points $\rho$.
{In particular}, any constraint of the form $\rho(f)=\alpha$ or $\rho(f)\leq \alpha$ with $f$ in $\cal{W}$ may be added to, or removed from, the definition of the feasible set $\L$. {The techniques of Section~\ref{sec:A} carry over if the constraints involve functions $f$ that are are non-negative or have finite supports.  Otherwise, the techniques of Appendix~\ref{shedvarfin} apply.}

%

\bibliographystyle{siamplain} 
\bibliography{cilpbib}

\appendix

\section{Proof of Theorem \ref{solvable}}\label{appendix}
First, we show that $\rho(f)$ is absolutely convergent for any $\rho$ in $\cal{L}$. Replacing $f$ by $\mmag{f}$ in \eqref{eq:markovthe1} we find that
\begin{align}\label{eq:dna7w89haw8feawff23}\rho(\mmag{f})&=\sum_{x\in\x_1}\mmag{f(x)}\rho(x)+\sum_{x\not\in\x_1}\mmag{f(x)}\rho(x)\\
&\leq \sum_{x\in\x_1}\mmag{f(x)}\rho(x)+c\left(\sup_{x\not\in\x_1}\frac{\mmag{f(x)}}{w(x)}\right)<\infty\nonumber\end{align}
where the right-most inequality follows from finiteness of $\x_1$ and our assumption that  $f$ belongs in $\cal{W}$.

Next, we show that \eqref{eq:boundconvthe1} is solvable (i.e.\ Theorem~\ref{solvable}$(ii,iii)$). To do so, we prove in Appendix~\ref{subapp} that the set $\cal{L}$ is weak$^\ast$ sequentially compact in the sense that every sequence of points $(\rho^r)_{r\in\zp}$ contained in $\cal{L}$ has a subsequence $(\rho^{r_k})_{r\in\zp}$ that converges in weak$^\ast$ (c.f. Definition~\ref{def:weak}) to a limit $\rho^\infty$ belonging to $\cal{L}$. We can then set $(\rho^r)_{r\in\zp}$ to be any sequence of feasible points of \eqref{eq:boundconvthe1} satisfying
$$\lim_{r\to\infty}\rho^r(f)=l_f.$$
Because $\cal{L}$ is weak$^\ast$ sequentially compact and $f$ belongs to $\cal{W}$, we can find a subsequence $(\rho^{r_k})_{r\in\zp}$ of $(\rho^r)_{r\in\zp}$ with a limit $\rho^\infty$ belonging to $\cal{L}$ that satisfies
$$\rho^\infty(f)=\lim_{k\to\infty}\rho^{r_k}(f)=l_f.$$
Setting $\rho_*:=\rho_\infty$ and replacing $\rho$ in \eqref{eq:dna7w89haw8feawff23} with $\rho_*$ then shows that \eqref{eq:boundconvthe1} is solvable.

\subsection{$\cal{L}$ is weak$^\ast$ sequentially compact}\label{subapp}
The argument consists of four steps:
\begin{enumerate}[label=(\alph*)]
\item Using a standard diagonal argument to find a pointwise convergent subsequence $(\rho^{r_k})_{k\in\zp}$ of $(\rho^r)_{r\in\zp}$.
\item Using $\rho^r(w)\leq c$ for all $r\in\zp$ and Fatou's lemma to show that the limit $\rho^\infty$ of $(\rho^{r_k})_{k\in\zp}$ is non-negative ($\rho^\infty(x)\geq0$ $\forall x\in\x$) and satisfies $\rho^\infty(w)\leq c$.
\item Using (a,b) and show that $(\rho^{r_k})_{k\in\zp}$ converges in weak$^*$ to $\rho^\infty$.
\item Using (c) to show that $\rho^\infty$ belongs to $\L$.
\end{enumerate}
\vspace{5pt}
Let's begin:
\vspace{5pt}
\begin{enumerate}[label=(\alph*)]
\item Enumerate the elements of $\x$ as $x_1,x_2,\dots$  Assumption~\ref{assumption}$(i,ii,iv)$ and the constraints $\rho H=\phi$, $\rho(g)=1$, and $\rho(w)\leq c$ in \eqref{eq:l2} imply that the sequence $(\rho^{r}(x_1))_{r\in\zp}$ is contained in a bounded interval. For this reason, the Bolzano-Weierstrass Theorem tells us that $(\rho^{r}(x_1))_{r\in\zp}$ has a converging subsequence $(\rho^{r_{k_1}}(x_1))_{{k_1}\in\zp}$. Repeating the same argument for $x_2$ and $(\rho^{r_{{k_1}}})_{{k_1}\in\zp}$, we obtain a convergent subsequence $(\rho^{r_{{k_2}}}(x_2))_{{k_2}\in\zp}$ of $(\rho^{r_{{k_1}}}(x_2))_{{k_1}\in\zp}$, and so on. Setting
$$\rho^{r_{k}}:=\rho^{r_{k_{k}}}\quad\forall k\in\zp,$$
we obtain the desired pointwise convergent subsequence.
\item Given (a), this follows directly from Fatou's lemma.
\item  For  any $f$ in $\cal{W}$ and $r'$ in $\zp$, 
\begin{align}\label{eq:together}&\mmag{\rho^{r_k}(f)-\rho^\infty(f)}\\
&\leq \sum_{x\in\x_{r'}}\mmag{\rho^{r_k}(x)-\rho^\infty(x)}\mmag{f(x)}+\sum_{x\not\in\x_{r'}}(\rho^{r_k}(x)+\rho^\infty(x))\mmag{f(x)}\nonumber.\end{align}
The generalisation~\eqref{eq:markovthe1} of Markov's inequality (with $\mmag{f}$ replacing $f$) and the moment bounds $\rho^\infty(w)\leq c$ and $\rho^{r_k}(w)\leq c$ tell us that
$$\sum_{x\not\in\x_{r'}}(\rho^{r_k}(x)+\rho^\infty(x))\mmag{f(x)}\leq  2c \left(\sup_{x\not\in\x_{r'}}\frac{\mmag{f(x)}}{w(x)}\right).$$
Fix $\varepsilon>0$. Because $f$ belongs to $\cal{W}$, we can find an $r'$ in $\zp$ such that $\sup_{x\not\in\x_{r'}}(\mmag{f(x)}/w(x))\leq \frac{\varepsilon}{4c}$. It follows from the above that
\begin{equation}\label{eq:tail}\sum_{x\not\in\x_{r'}}(\rho^{r_k}(x)+\rho^\infty(x))\mmag{f(x)}\leq  \frac{\varepsilon}{2}.\end{equation}
Because $\x_{r'}$ is a finite set ($w$ is norm-like), the pointwise convergence of $\rho^{r_k}$ to $\rho^\infty$ implies that there exist a $K$ such that
\begin{equation}\label{eq:nontail}\sum_{x\in\x_{r'}}\mmag{\rho^{r_k}(x)-\rho^\infty(x)}\mmag{f(x)}\leq \frac{\varepsilon}{2}\quad\forall k\geq K.\end{equation}
Combining \eqref{eq:together}--\eqref{eq:nontail} yields
$$\mmag{\rho^{r_k}(f)-\rho^\infty(f)}\leq\varepsilon\quad\forall k\geq K.$$
Because the $\varepsilon$ was arbitrary, we have the desired limit:
$$\lim_{k\to\infty}\rho^{r_k}(f)=\rho^\infty(f).$$
Since the above holds for every $f\in\cal{W}$, we have that $\rho^{r_k}$ not only converges pointwise to $\rho$ but also in weak$^*$. 
\item Assumption~\ref{assumption}$(ii,iii,v)$ implies that $g$ and $x'\mapsto h(x',x)$ (for any $x\in\x$) belong to $\cal{W}$. For this reason, the weak$^*$ convergence of the subsequence implies that
$$\rho^\infty H(x)=\lim_{k\to\infty}\rho^{r_k} H(x)=\phi(x)\quad\forall x\in\x,\qquad \rho^\infty(g)=\lim_{k\to\infty}\rho^{r_k}(g)=1. $$
Because we already argued in (b) that $\rho^\infty(x)\geq0$ for all $x$ in $\x$ and $\rho^\infty(w)\leq c$, the above shows that that $\rho^\infty$ belongs to $\L$.
\end{enumerate}

\section{Weak$^\ast$ convergence implies convergence in total variation}\label{wtv}
The argument is as follows: given our identification \eqref{eq:not1} of $\ell^1$ with the space of finite signed measures on $(\x,\twx)$, the total variation norm $\norm{\cdot}$ on $\ell^1$ {is} dominated by (in fact, equivalent to)  the $\ell^1$-norm:
$$\norm{\rho}\leq\sum_{x\in\s}\mmag{\rho(x)}=:\norm{\rho}_1\quad\forall\rho\in\ell^1.$$
Because the dual of $(\ell^1,\norm{\cdot}_1)$ is isomorphic to the space of bounded functions on $\x$, Schur's Theorem tells us that a sequence $(\rho^r)_{r\in\zp}\subseteq \ell^1$ converges in the $\ell^1$-norm to $\rho\in\ell^1$ if and only if \eqref{eq:weakstardef} holds for all bounded functions $f$. Because the norm-like assumption implies that all bounded functions belong to $\cal{W}$, it follows that weak$^\ast$ convergence implies convergence in $\ell^1$ and, consequently, in total variation.

\section{Shedding Assumption \ref{assumption}$(v)$}\label{shedvarfin}
While Assumptions~\ref{assumption}$(i$--$iv)$ are very mild and widely satisfied in applications, Assumption~\ref{assumption}$(v)$ requiring that every column of $H$ has finitely many non-zero entries is more restrictive. Indeed, it is not exceedingly rare for a discrete-time chain to possess one or more states that are reachable from infinitely many other states in a single step, in which case the one-step matrix $P$ would violate Assumption~\ref{assumption}$(v)$.
Although Assumption~\ref{assumption}$(v)$ does simplify the content of this paper, it can be circumvented as follows.

The assumption is important for two reasons: 
(a) practically, it ensures that, for each $x$ in $\x$, the equation 
$$\phi(x) = \rho H(x)=\sum_{x'\in\x} \rho(x')h(x',x)=\sum_{x'\in\supp{h(\cdot,x)}}\rho(x')h(x',x)$$
involves only finitely many entries of $\rho$, a key fact in the implementation of our schemes;  and 
(b) theoretically, it ensures that, for each $x$ in $\x$, the function $x'\mapsto h(x',x)$ belongs to $\cal{W}$ in \eqref{eq:grs}. For this reason, the weak$^*$ convergence of the approximating sequences $(\rho^r)_{r\in\zp}$ of Section~\ref{sec:A} guarantees that their limit points $\rho^\infty$ satisfy the equations $\rho H=\phi$, a fact necessary when showing that the limit points belonging to $\L$ (see the proof of Theorem~\ref{convlpt} for details).

This assumption can be avoided using the moment bound \eqref{eq:ggmomb} at the expense of doubling the number of decision variables in our finite-dimensional LPs. To do so, we require a sequence $b_1,b_2,\dots$ of known constants such that
\begin{equation}\label{eq:hbounds}\sup_{x\not\in\x_r}\frac{\sum_{z\in\x_r}h(x,z)}{w(x)}\leq b_r\quad\forall r\in\zp,\quad\lim_{r\to\infty}b_r=0.\end{equation} 
The above implies that $x'\mapsto h(x',x)$ belongs to $\cal{W}$ for each $x\in\x$ and we recover (b). 
To recover finite-dimensional outer approximations of $\L$ of the sort in Section~\ref{sec:A}, pick any index $x$ in our truncation $\x_r$ and consider its associated equation:
\begin{equation}\label{eq:modeqsp}\rho H(x)=\sum_{x'\in\x_r}\rho(x')h(x',x)+\epsilon_r(x)=\phi(x),\end{equation}
where the measure $\epsilon_r$ is defined by
$$\epsilon_r(x):=\left\{\begin{array}{ll}\sum_{x'\not\in\x_r}\rho(x')h(x',x)&\text{if }x\in\x_r\\
0&\text{if }x\not\in\x_r\end{array}\right.\quad\forall x\in\x.$$
Because $H$ is Metzler and any feasible point $\rho$ of $\L$ is non-negative, we have that
\begin{equation}\label{eq:epcv}\epsilon_r(x)\geq0\quad\forall x\in\x.\end{equation}
Tonelli's Theorem, the generalisation~\eqref{eq:markovthe1} of Markov's inequality, and \eqref{eq:hbounds} imply
\begin{align}\epsilon_r(\x_r)&=\sum_{z\in\x_r}\epsilon_r(z)=\sum_{x\not\in\x_r}\rho(x)\left(\sum_{z\in\x_r}h(x,z)\right)\leq b_r\left(\sum_{x\not\in\x_r}\rho(x)w(x)\right)\nonumber\\
&\leq b_r\rho(w)\leq cb_r,\qquad \forall \rho\in\L.\label{eq:epbouns}\end{align}
Putting \eqref{eq:rhormomb}, \eqref{eq:gconstrel}, and \eqref{eq:modeqsp}--\eqref{eq:epbouns} together, we recover a finite dimensional outer approximation of $\L$: if $\rho$ belongs to $\L$, then the pair $(\rho_{|r},\epsilon_r)$ belongs to 
$$\tilde{\L}_r:=\left\{(\rho^r,\epsilon^r)\in\ell^1\times\ell^1: \begin{array}{l} \rho^r H(x)+\epsilon^r(x)=\phi(x),\quad\forall x\in\x_r,\\ 1-ca_r\leq\rho^r(g)\leq 1, \\ \epsilon^r(\x_r)\leq cb_r,\quad\rho^r(w)\leq c,\\ \rho^r(x)\geq0,\quad \epsilon^r(x)\geq0,\quad\forall x\in\x, \\\rho^r(\x_r^c)+\epsilon^r(\x_r^c)=0. \end{array}\right\}.$$
For this reason, replacing Assumption~\ref{assumption}$(v)$ with `there exists a known sequence $(b_r)_{r\in\zp}$ such that \eqref{eq:hbounds} holds', $\L_r$ with $\tilde{\L}_r$, and $l^r_f,u^r_f$ in \eqref{eq:favet1}--\eqref{eq:favet2} with
$$\tilde{l}^r_f:=\inf\{{\rho^r}(f):(\rho^r,\epsilon^r)\in\tilde{\L}_r\},\qquad \tilde{u}^r_f:=\inf\{{\rho^r}(f):(\rho^r,\epsilon^r)\in\tilde{\L}_r\},$$
the results of Sections \ref{sec:A}--\ref{sec:B} hold identically. The only difference is that in contrast with $\L_r$ in \eqref{eq:lpap}, $\tilde{\L}_r$  involves $2|\x_r|$ variables decision variables, $|\x_r|$ equalities, and $2|\x_r|+4$ inequalities. 

\end{document}